\documentclass[10pt]{amsart}
\usepackage{amsfonts,amsthm,amsmath,color,graphics,hyperref}
\usepackage[all]{xy}
\usepackage{indentfirst}
\newtheorem{theorem}{Theorem}[section]
\newtheorem{corollary}[theorem]{Corollary}
\newtheorem{definition}[theorem]{Definition}
\newtheorem{lemma}[theorem]{Lemma}
\newtheorem{proposition}[theorem]{Proposition}
\newtheorem{remark}[theorem]{Remark}

\newtheorem{example}[theorem]{Example}
\newtheorem{assumption}[theorem]{Assumption}

\begin{document}
\title{Weierstrass filtration on Teichm\"{u}ller curves and Lyapunov exponents: Upper bounds}
\author{Fei Yu}
\author{Kang Zuo}
\address{School of Mathematical Sciences, Zhejiang University, China}
\address{School of Mathematical Sciences, Xiamen University, China}
\email{yufei@zju.edu.cn}
\address{Universit{\"a}t Mainz, Fachbereich 17, Mathematik, 55099 Mainz, Germany}
\email{zuok@uni-mainz.de}
\date{\today. \\The research of F.Yu was supported by is supported by the Fundamental Research Funds for the Central Universities (No.
20720140526).  The work of K.Zuo is funded by the SFB/TR 45 ¡®Periods, Moduli Spaces and Arithmetic of Algebraic
Varieties¡¯ of the DFG}

\maketitle
\begin{abstract}We get an
upper bound of the slope of each graded quotient for the
Harder-Narasimhan filtration of the Hodge bundle of a
Teichm\"{u}ller curve. As an application, we show that the sum of
Lyapunov exponents of a Teichm\"{u}ller curve does not exceed
${(g+1)}/{2}$, with equality reached if and only if the curve lies
in the hyperelliptic locus induced from
$\mathcal{Q}(2k_1,...,2k_n,-1^{2g+2})$ or it is a special
Teichm\"{u}ller curve in $\Omega\mathcal{M}_g(1^{2g-2})$. It also gives an unified interpretation for many known results about the special partial sums of Lyapunov exponents on Teichm\"uller curves.
\end{abstract}
 \tableofcontents
\section{Introduction}
Let $\mathcal{M}_g$ be the moduli space of Riemann surfaces of genus
$g$, and $\Omega\mathcal{M}_g\rightarrow \mathcal{M}_g$ the bundle
of pairs $(X,\omega)$, where $\omega\neq 0$ is a holomorphic 1-form
on $X\in \mathcal{M}_g$. Denote by
$\Omega\mathcal{M}_g(m_1,...,m_k)$ the stratum of pairs $(X,\omega)$, where
$\omega(\neq 0)$ have $k$ distinct zeros of order $m_1,...,m_k$
respectively.

There is a nature action of $GL_2^+(\mathbb{R})$ on
$\Omega\mathcal{M}_g(m_1,...,m_k)$, whose orbits project to complex
geodesics in $\mathcal{M}_g$. The projection of an orbit is almost
always dense. However, if the stabilizer $SL(X,\omega)\subset
SL_2(\mathbb{R})$ of a given form is a lattice,   then the
projection of its orbit gives a closed, algebraic Teichm\"{u}ller
curve $C$.

After suitable base change and compacfication, we can get a
universal family $f:S\rightarrow C$, which is a relative minimal
semistable model with disjoint sections $D_1,...,D_k$; here $D_i|_X$
is a zero of $\omega$ when restrict to each fiber $X$.

The relative canonical bundle formula \eqref{canonical} of the
Teichm\"{u}ller curve is (\cite{CM12}\cite{EKZ}):
$$\omega_{S/C}\simeq f^*\mathcal{L}\otimes \mathcal{O}(\sum^{k}_{i=1} m_i
D_i).$$

Here $\mathcal{L}\subset f_*{\omega_{S/C}}$ be the   line
bundle whose fiber over the point corresponding to $X$ is
$\mathbb{C}\omega$, the generating differential of Teichm\"{u}ller
curves and
$$\mathrm{deg}(\mathcal{L})=(2g(C)-2+|\Delta|)/2.$$

There are many nature  vector subbundles of the Hodge bundle
$f_*(\omega_{S/C})$:
$$\mathcal{L}\otimes f_*\mathcal{O}(\sum^{k}_{i=1} d_i D_i)\subset
\mathcal{L}\otimes f_*\mathcal{O}(\sum^{k}_{i=1} m_i D_i)=f_*(\omega_{S/C}).$$
One can construct many filtration (We call them Weierstrass filtration in \cite{YZ} because their structures are closely related with the Weierstrass semigroup of points of general fibers) by using these subbundles. In
particular, using properties of Weierstrass semigroups, we have
constructed the Harder-Narasimhan filtration of $f_*(\omega_{S/C})$
for Teichm\"{u}ller curves in hyperelliptic loci and some low genus
nonvarying strata \cite{YZ}. In this article, we will get an upper
bound of the slope of each graded quotient for the Harder-Narasimhan
filtration of $f_*(\omega_{S/C})$ of Teichm\"{u}ller curves in each
stratum.

For a vector bundle $V$, let $\mathrm{gr}^{HN}_j$ be the $j$- graded part of Harder-Narasimhan filtration of $V$. If
$\mathrm{rk}(HN_{j-1}(V))<i\leq \mathrm{rk}(HN_j(V))$, we denote by $\mu_i(V)$ the slope $\mu(\mathrm{gr}^{HN}_j)$. Write $w_i$ for
$\mu_i(f_*(\omega_{S/C}))/\mathrm{deg}(\mathcal{L})$, then we have:
\begin{theorem}[Theorem \ref{ubhn}] For a Teichm\"{u}ller
curve which lies in $\Omega\mathcal{M}_g(m_1,...,m_k)$. We order the numbers in the set
$\{\frac{j}{m_l+1}|1\leq j\leq m_l,1\leq l\leq k\}$
increasing(counted with multiplicity) as $a_1\leq a_2\leq \cdot\cdot\cdot \leq a_{m_1+\cdot\cdot\cdot+m_k}$.

Then we have $w_1=1$. For $2 \leq i\leq g$, there are
inequalities:
$$w_i\leq 1-a_{H_i(P)},$$
here $P$ is the special permutation \eqref{arrange} which satisfies $H_i(P)\geq 2i-2$.
\end{theorem}
We will use the example \ref{ex} to explain this theorem.

Fix an $SL_2(\mathbb{R})$-invariant, ergodic measure $\mu$ on
$\Omega\mathcal{M}_g$. The Lyapunov exponents for the
Teichm\"{u}ller geodesic flow on $\Omega\mathcal{M}_g$ measure the
logarithm of the growth rate of the Hodge norm of cohomology classes
under the parallel transport along the geodesic flow.

In general,  it is difficult to compute the Lyapunov exponents.
There are some algebraic attempts to compute some special partial sums of
Lyapunov exponents, all of which are based on the following fact: in these cases,
the special partial sums of the Lyapunov exponents is related with the degree of
certain vector bundles (cf. Theorem \ref{sumly}). In particular, the
sum of Lyapunov exponents of a Teichm\"{u}ller curve equals
$\mathrm{deg}(f_*(\omega_{S/C}))/\mathrm{deg}(\mathcal{L})$. This algebraic
interpretation combined with  information about the
Harder-Narasimhan filtration gives us the following estimate:

\begin{theorem}[Theorem \ref{main}] The sum of Lyapunov exponents of a Teichm\"{u}ller curve in $\Omega\mathcal{M}_g(m_1,...,m_k)$ satisfies the inequality
$$L(C)\leq \frac{g+1}{2}.$$
Furthermore, equality occurs if and only if it lies in the
hyperelliptic locus induced from
$\mathcal{Q}(2k_1,...,2k_n,-1^{2g+2})$ or it is a special
Teichm\"{u}ller curve in $\Omega\mathcal{M}_g(1^{2g-2})$.
\end{theorem}
Dawei Chen and Martin M\"{o}ller have obtained many interesting upper
bounds in \cite{CM12}\cite{Mo13}. When $k\geq 4$, this upper bound better than the upper bound obtained by using Cornalba-Harris-Xiao's slope inequality.

The Harder-Narasimhan filtration also gives rise to an upper bound
of the degrees of any locally free subsheaf, especially those
related to the special partial sums of the Lyapunov exponents.

\begin{proposition}[Proposition \ref{pa}]If the VHS over the Teichm\"uller curve $C$ contains a sub-VHS $\mathbb{W}$
of rank $2k$, then the sum of the $k$ corresponding non-negative
Lyapunov exponents is the sum of  $w_{i_1},...,w_{i_k}$ (where $i_j$
are different to each other) and satisfies
$$\overset{k}{\underset{i=1}{\sum}}\lambda^{\mathbb{W}}_i\leq 1+\sum^k_{i=2}(1-a_{H_i(P)}).$$
\end{proposition}

For individual Lyapunov exponents, due to the lack of algebraic
interpretation\footnote{Base on the result of this paper, the first author make a conjecture that partial sums of Lyapunov exponents might be estimated through Chern classes
of holomorphic vector bundles over Teichm\"uller curves normalized by the Euler
characteristics of Teichm\"uller curves \cite{Yu}. Recently a proof of this conjecture is announced by Eskin-Kontsevich-M\"oller-Zorich.}, we will make the following assumption:

\begin{assumption}[Assumption \ref{assumption}]$f_*(\omega_{S/C})$ equals $(\overset{h}{\underset{i=1}{\oplus}} L_i) \oplus W$, here $L_i$ are line bundles
such that the $i$-th Lyapunov exponent satisfies the equality:
$$\lambda_i=\Big\{
 \begin{array}{cc}
\mathrm{deg}(L_i)/\mathrm{deg}(\mathcal{L}) &1\leq i\leq h \\
0 &h< i\leq g.
   \end{array}
$$
\end{assumption}
There are many examples satisfying this assumption: triangle groups
\cite{BM10}\cite{Wr2}, square tiled cyclic covers \cite{EKZ11}\cite{FMZ},
square tiled abelian covers \cite{Wr1}, some wind-tree models
\cite{DHL}, and algebraic primitives\cite{BM09}.

Our estimate on the slopes of the Harder-Narasimhan filtration will
give the following upper bound for individual Lyapunov exponents:

\begin{proposition}[Proposition \ref{single}] For a Teichm\"{u}ller curve which
satisfies the assumption \ref{assumption} and lies in
$\Omega\mathcal{M}_g(m_1,...,m_k)$. We order the numbers in the set
$\{\frac{j}{m_l+1}|1\leq j\leq m_l,1\leq l\leq k\}$
increasing(counted with multiplicity) as $a_1\leq a_2\leq \cdot\cdot\cdot \leq a_{m_1+\cdot\cdot\cdot+m_k}$.

For $2 \leq i\leq g$, the $i$-th Lyapunov exponent
satisfies the inequality:
$$\lambda_i \leq 1-a_{H_i(P)},$$
here $P$ is the special permutation \eqref{arrange} which satisfies $H_i(P)\geq 2i-2$.
\end{proposition}
The equality can be reached for an algebraic primitive
Teichm\"{u}ller curve lying in the hyperelliptic locus induced from
$\mathcal{Q}(2k_1,...,2k_n,-1^{2g+2})$.

 For Teichm\"{u}ller curves lying in hyperelliptic loci
and some low genus nonvarying strata, the following proposition is
obvious because we have constructed the Harder-Narasimhan filtration
in \cite{YZ}.

\begin{proposition}[Proposition \ref{nonva}] For a Teichm\"{u}ller curve which
satisfies the assumption \ref{assumption} and lies in hyperelliptic
loci or one of the following strata:\\
$\overline{\Omega\mathcal{M}}_3(4),\overline{\Omega\mathcal{M}}_3(3,1),\overline{\Omega\mathcal{M}}^{odd}_3(2,2),\overline{\Omega\mathcal{M}}_3(2,1,1)$\\
$\overline{\Omega\mathcal{M}}_4(6),\overline{\Omega\mathcal{M}}_4(5,1),\overline{\Omega\mathcal{M}}^{odd}_4(4,2),\overline{\Omega\mathcal{M}}^{non-hyp}_4(3,3),\overline{\Omega\mathcal{M}}^{odd}_4(2,2,2),\overline{\Omega\mathcal{M}}_4(3,2,1)$\\
$\overline{\Omega\mathcal{M}}_5(8),\overline{\Omega\mathcal{M}}_5(5,3),\overline{\Omega\mathcal{M}}^{odd}_5(6,2)$\\
 The $i$-th Lyapunov exponent $\lambda_i$ equals the $w_i$ which is computed in the Theorem \ref{YZ}.
\end{proposition}
\paragraph{\textbf{Acknowledgement}}We thank Ke Chen for a his many suggestions.
\section{Harder-Narasimhan
filtration}

The readers are referred to \cite{HL} for details about sheaves on
algebraic varieties. Let $C$ be a smooth projective curve, $V$ a
vector bundle over $C$ of slope $\mu(V):=\frac{\mathrm{deg}(V)}{\mathrm{rk}(V)}$. We
call $V$ semistable (resp.stable) if $\mu(W)\leq\mu(V)$
(resp.$\mu(W)<\mu(V)$) for any subbundle $W\subset V$. If $V_1,V_2$
are semistable such that $\mu(V_1)>\mu(V_2)$, then any map
$V_1\rightarrow V_2$ is zero.

A Harder-Narasimhan filtration for $V$ is an increasing filtration:
$$0=HN_0(V)\subset HN_1(V)\subset \cdot\cdot\cdot\subset HN_k(V)$$
such that the graded quotients $\mathrm{gr}^{HN}_i=HN_i(V)/HN_{i-1}(V)$ for
$i=1,...,k$ are semistable vector bundles and
$$\mu(\mathrm{gr}^{HN}_1)>\mu(\mathrm{gr}^{HN}_2)>\cdot\cdot\cdot>\mu(\mathrm{gr}^{HN}_k)$$
The Harder-Narasimhan filtration is unique.

A Jordan-H\"{o}lder filtration for a semistable vector bundle $V$ is a
filtration:
$$0=V_0\subset V_1\subset \cdot\cdot\cdot\subset V_l=V$$
such that the graded quotients $\mathrm{gr}^{V}_i=V_i/V_{i-1}$ are stable of
the same slope.

A Jordan-H\"{o}lder filtration always exist, but it is not unique in
general. The graded objects $\mathrm{gr}^{V}=\oplus \mathrm{gr}^{V}_i$ of a semistable vector bundle $V$ do  not
depend on the choice of the Jordan-H\"{o}lder filtration.

For a rank $r$ vector bundle $V$,  If
$\mathrm{rk}(HN_{j-1}(V))<i\leq \mathrm{rk}(HN_j(V))$, we denote by $\mu_i(V)$ the slope $\mu(\mathrm{gr}^{HN}_j)$ . Obviously we have $\mu_1(V)\geq
\cdot\cdot\cdot\geq \mu_r(V)$.

\begin{lemma}\label{control} Let $V$ and $U$ be two vector bundles of rank $n$ over $C$, with increasing filtration
$$0=V_0\subset V_1\subset \cdot\cdot\cdot\subset V_n=V$$
$$0=U_0\subset U_1\subset \cdot\cdot\cdot\subset U_n=U$$
such  that $V_i/V_{i-1},U_i/U_{i-1}$ are line bundles,
$V_i/V_{i-1}\subset U_i/U_{i-1}$ and the degrees $\mathrm{deg}(U_i/U_{i-1})$
decrease in $i$ ($1\leq i\leq n$). Then $\mu_i(V)\leq
\mathrm{deg}(U_i/U_{i-1})$.
\end{lemma}
\begin{proof}If there is some $\mu_i(V)$ bigger than $
\mathrm{deg}(U_i/U_{i-1})$, where $\mu_i(V)$ equals $\mu(\mathrm{gr}^{HN(V)}_j)$, then for $l\geq i$,
$$\mu_i(V)>\mathrm{deg}(U_i/U_{i-1})\geq \mathrm{deg}(U_l/U_{l-1})\geq
\mathrm{deg}(V_l/V_{l-1}).$$

For $m\leq j,l\geq i$, the quotients $\mathrm{gr}^{HN(V)}_m$ and $V_l/V_{l-1}$ are
semistable. $$\mu(\mathrm{gr}^{HN(V)}_m)\geq\mu(\mathrm{gr}^{HN(V)}_j)=\mu_i(V)>\mathrm{deg}(V_l/V_{l-1}),$$ so
any map $\mathrm{gr}^{HN(V)}_m\rightarrow V_l/V_{l-1}$ is zero. Thus any map
$\mathrm{gr}^{HN(V)}_m\rightarrow V/V_{i-1}$ is zero by induction on $l$, and
any map $HN_j(V) \rightarrow V/V_{i-1}$ is zero by induction on $m$.

So the the canonical morphism $HN_j(V)\hookrightarrow
V\rightarrow V/V_{i-1}$ is zero, that is $HN_j(V)\hookrightarrow
V_{i-1}$, which is a contradiction because $\mathrm{rk}(HN_j(V))\geq
i>i-1=\mathrm{rk}(V_{i-1})$.
\end{proof}
Let $\mathrm{grad}(HN(V))$ denote the direct sum of the graded quotients of
the Harder-Narasimhan filtration, that is $\mathrm{grad}(HN(V))=\oplus \mathrm{gr}^{HN(V)}_i$.
\begin{lemma}[\cite{YZ}]\label{directsum}For vector bundles $V_1,...,V_n$, we have: $$\mathrm{grad}(HN(V_1\oplus\cdots\oplus V_n))=\mathrm{grad}(HN(V_1))\oplus\cdots\oplus \mathrm{grad}(HN(V_n))$$
and $\mu_i(V_j)$   equals $\mu_k(V_1\oplus\cdots\oplus V_n)$ for some
$k$.
\end{lemma}

\section{Filtration of the Hodge bundle} Let
$\Omega\mathcal{M}_g(m_1,...,m_k)$ be the stratum parameterizing
$(X,\omega)$ where $X$ is a curve of genus $g$ and $\omega$ is an
Abelian differentials (i.e.a holomorphic one-form) on $X$ that have
$k$ distinct zeros of order $m_1,...,m_k$. Denote by
$\overline{\Omega\mathcal{M}}_g(m_1,...,m_k)$ the Deligne-Mumford
compactification of $\Omega\mathcal{M}_g(m_1,...,m_k)$.  Denote by
$\Omega\mathcal{M}^{hyp}_g(m_1,...,m_k)$( resp. odd, resp. even) the
hyperelliptic (resp.  odd theta character, resp. even theta
character) connected component. (\cite{KZ03})

Let $\mathcal{Q}(d_1,...,d_n)$ be the stratum parameterizing $(X,q)$
where $X$ is a curve of genus $g$ and $q$ is a meromorphic quadratic
differentials with at most simple zeros on $X$ that have $k$
distinct zeros of order $d_1,...,d_n$ respectively.

If the quadratic differential is not a global square of a one-form,
there is a canonical double covering $\pi:Y\rightarrow X$ such that
$\pi^*q=\omega^2$. This covering is ramified precisely at the zeros
of odd order of $q$ and at the poles. It give a map
$$\phi:\mathcal{Q}(d_1,...,d_n)\rightarrow\Omega\mathcal{M}_g(m_1,...,m_k).$$
A singularity of order $d_i$ of $q$ give rise to two zeros of degree
$m=d_i/2$ when $d_i$ is even, single zero of degree $m=d+1$ when $d$
is odd. Especially, the hyperelliptic locus in a stratum
$\Omega\mathcal{M}_g(m_1,...,m_k)$ induces from a stratum
$\mathcal{Q}(d_1,...,d_k)$ satisfying $d_1+\cdots+d_n=-4$.

There is a nature action of $GL_2^+(\mathbb{R})$ on
$\Omega\mathcal{M}_g(m_1,...,m_k)$, whose orbits project to complex
geodesics in $\mathcal{M}_g$. The projection of an orbit is almost
always dense. If the stabilizer $SL(X,\omega)\subset
SL_2(\mathbb{R})$ of given form is a lattice, however, then the
projection of its orbit gives a closed, algebraic Teichm\"{u}ller
curve $C$.

 The Teichm\"{u}ller curve $C$ is an
algebraic curve in $\overline{\Omega\mathcal{M}}_g$ that is totally
geodesic with respect to the Teichm\"{u}ller metric.

After suitable base change, we can get a universal family
$f:S\rightarrow C$, which is a relatively minimal semistable model
with disjoint sections $D_1,...,D_k$; here $D_i|_X$ is a zero of
$\omega$ when restrict to each fiber $X$. (\cite{CM12})

Let $\mathcal{L}\subset f_*{\omega_{S/C}}$ be the   line bundle
whose fiber over the point corresponding to $X$ is
$\mathbb{C}\omega$, the generating differential of Teichm\"{u}ller
curves; it is also known as  the "maximal Higgs" line bundle. Let
$\Delta\subset \overline{B}$ be the set of points with singular
fibers, then the property of being ''maximal Higgs'' says by
definition that $\mathcal{L}\cong
\mathcal{L}^{-1}\otimes\omega_C(\mathrm{log}{\Delta})$ and
$$\mathrm{deg}(\mathcal{L})=(2g(C)-2+|\Delta|)/2,$$
together with an identification (relative canonical bundle formula
\cite{CM12} \cite{EKZ}):
\begin{equation}\label{canonical}
 \omega_{S/C}\simeq f^*\mathcal{L}\otimes \mathcal{O}(\sum^k_{i=1} m_i D_i).
\end{equation}
By the adjunction formula we get
$$D^2_i=-\omega_{S/C}D_i=-m_iD^2_i-\mathrm{deg}\,{\mathcal{L}},$$
and thus
\begin{equation}\label{intersection}
 D^2_i=-\frac{1}{m_i+1}\mathrm{deg}\, \mathcal{L}.
\end{equation}
For a line bundle $\mathcal{L}$ of degree $d$ on $X$, denote by
$h^0(\mathcal{L})$ the dimension of $H^0(X,\mathcal{L})$. From
the exact sequence
$$0\rightarrow f_*\mathcal{O}(d_1D_1+\cdots+d_kD_k)\rightarrow f_*\mathcal{O}(m_1D_1+\cdots+m_kD_k)=f_*(\omega_{S/C})\otimes\mathcal{L}^{-1},$$
and the fact that all subsheaves of a locally free sheaf on a curve
are locally free, we deduce that $f_*\mathcal{O}(d_1D_1+\cdots+d_kD_k)$
is a vector bundle of rank $h^0(d_1p_1+\cdots+d_kp_k)$, here
$p_i=D_i|_F$, $F$ is a generic fiber. We have constructed many
filtration of the Hodge bundle by using those vector bundles in \cite{YZ}(We call them Weierstrass filtration in \cite{YZ} because their structures are closely related with the Weierstrass semigroup of points of general fibers).

A fundament exact sequence for those filtration is the following:
\begin{equation}\label{basic}
0\rightarrow f_*\mathcal{O}(\sum (d_i-a_i)D_i)\rightarrow
f_*\mathcal{O}(\sum d_iD_i)\rightarrow f_*\mathcal{O}_{\sum
a_iD_i}(\sum d_iD_i)\overset{\delta}{\rightarrow}
\end{equation}
$$R^1f_*\mathcal{O}(\sum (d_i-a_i)D_i)\rightarrow
R^1f_*\mathcal{O}(\sum d_iD_i)\rightarrow 0.$$

There are many properties of these filtration:
\begin{lemma}[\cite{YZ}]\label{EE}
If $h^0(\sum d_ip_i)=h^0(\sum (d_i-a_i)p_i)$ holds in a general
fiber, then we have the equality $f_*\mathcal{O}(\sum
d_iD_i)=f_*\mathcal{O}(\sum (d_i-a_i)D_i)$.
\end{lemma}
\begin{lemma}[\cite{YZ}]\label{SP}
If $h^0(\sum d_ip_i)=h^0(\sum (d_i-a_i)p_i)+\sum a_i$ is
non-varying, then
$$f_*\mathcal{O}(\sum d_iD_i)/f_*\mathcal{O}(\sum (d_i-a_i)D_i)=f_*\mathcal{O}_{\sum a_iD_i}(\sum d_iD_i)=\bigoplus f_*\mathcal{O}_{a_iD_i}(d_iD_i).$$
\end{lemma}
\begin{lemma}[\cite{YZ}]\label{HN}
The Harder-Narasimhan filtration of $f_*\mathcal{O}_{aD}(dD)$ is
$$0\subset f_*\mathcal{O}_{D}((d-a+1)D)\subset \cdots\subset f_*\mathcal{O}_{(a-1)D}((d-1)D)\subset f_*\mathcal{O}_{aD}(dD),$$
and the direct sum of the graded quotient of this filtration is
$$\mathrm{grad}(HN(f_*\mathcal{O}_{aD}(dD)))=\overset{a-1}{\underset{i=0}{\bigoplus}}
\mathcal{O}_{D}((d-i)D).$$
\end{lemma}
\begin{lemma}[\cite{YZ}]\label{UP}Let $V=f_*\mathcal{O}(\sum d_iD_i)/f_*\mathcal{O}(\sum (d_i-a_i)D_i))$ and $r=h^0(\sum
d_ip_i)-h^0(\sum (d_i-a_i)p_i)$, where $p_i$ is $D_i|_F$ for a general fiber $F$. We order degrees of line bundles in the
set
$$\{\mathcal{O}_{D_i}((d_i-j)D_i)|1\leq i\leq k, 0\leq j\leq a_i-1\},$$
decreasing(counted with multiplicity) as $b_1\geq b_2\geq\cdot\cdot\cdot\geq b_{a_1+\cdot\cdot\cdot+a_k}$. Then
$$\mathrm{deg}(V)\leq b_1+b_2+\cdot\cdot\cdot+b_r.$$
\end{lemma}
For a Teichm\"{u}ller curves lying in hyperelliptic loci and low
genus non-varying strata, we have constructed the Harder-Narasimhan
filtration.

Write $w_i$ for $\mu_i(f_*(\omega_{S/C}))/\mathrm{deg}(\mathcal{L})$.

\begin{theorem}[\cite{YZ}] \label{YZ}Let $C$ be a Teichm\"{u}ller curve in the hyperelliptic locus of some
stratum $\overline{\Omega\mathcal{M}}_g(m_1,...,m_k)$, and denote by
$(d_1,...,d_n)$ the orders of singularities of underlying quadratic
differentials. Then $w_i$ for $C$ is the $i$-th largest number in
the set
$$\{1\}\cup\Big\{1-\frac{2k}{d_j+2}\Big\}_{\forall d_j, 0<2k\leq d_j+1}$$
For  a Teichm\"{u}ller curve lying in some low genus non varying
strata, the $w_i$'s are computed in Table 1, 2 and 3.
\end{theorem}
\begin{table}
\caption{genus 3}
\begin{tabular}{|c|c|c|c|c|c|}
  \hline
  zeros & component & $w_2$& $w_3$ & $\sum w_i$  \\
   \hline
  (4)& hyp & 3/5 & 1/5 & 9/5  \\
    \hline
  (4) & odd & 2/5 & 1/5 & 8/5  \\
  \hline
   (3,1) &   & 2/4 & 1/4 & 7/4  \\
  \hline
   (2,2) & hyp & 2/3 & 1/3 & 2  \\
  \hline
   (2,2) & odd & 1/3 & 1/3 & 5/3  \\
  \hline
   (2,1,1) &   & 1/2 & 1/3 & 11/6  \\
  \hline
  (1,1,1,1) &   &   &   &   $\leq 2$ \\
  \hline
\end{tabular}
\end{table}
\begin{table}
\caption{genus 4}
 \begin{tabular}{|c|c|c|c|c|c|c|}
  \hline
  zeros & component & $w_2$& $w_3$ &$w_4$ & $\sum w_i$  \\
   \hline
  (6)&  hyp & 5/7 & 3/7 & 1/7& 16/7 \\
    \hline
  (6) & even& 4/7 & 2/7 & 1/7 &14/7   \\
  \hline
  (6) & odd & 3/7 & 2/7 & 1/7 &13/7   \\
  \hline
   (5,1) &  & 1/2 & 2/6 & 1/6 & 2   \\
   \hline
   (3,3) & hyp & 3/4 & 2/4 & 1/4 & 5/2  \\
   \hline
   (3,3) & non-hyp & 2/4 & 1/4 & 1/4 & 2  \\
  \hline
     (4,2) & even & 3/5 & 1/3 & 1/5 &32/15   \\
  \hline
   (4,2) & odd & 2/5 & 1/3 & 1/5 &29/15   \\
  \hline
  (2,2,2) &   & 1/3 & 1/3 & 1/3 &2 \\
  \hline
     (3,2,1) &   & 1/2 & 1/3 & 1/4 &25/12   \\
  \hline
\end{tabular}
\end{table}

\begin{table}
\caption{genus 5}
\begin{tabular}{|c|c|c|c|c|c|c|c|}
  \hline
  zeros & component & $w_2$& $w_3$ &$w_4$ &$w_5$ & $\sum w_i$  \\
   \hline
  (8)&  hyp & 7/9 & 5/9 & 3/9 &1/9& 25/9 \\
    \hline
  (8) & even& 5/9 & 3/9 & 2/9 &1/9&20/9   \\
  \hline
  (8) & odd & 4/9 & 3/9 & 2/9 &1/9&19/9   \\
  \hline
   (5,3) &   & 1/2 & 1/3 & 1/4 &1/6& 9/4   \\
  \hline
      (6,2) & odd  & 3/7 & 1/3 & 2/7 &1/7& 46/21  \\
  \hline
     (4,4) & hyp  & 4/5 & 3/5 & 2/5 &1/5& 3  \\
  \hline
\end{tabular}
\end{table}
\section{Lyapunov exponents}
A good introduction to Lyapunov exponents with a lot of motivating
examples is the survey by Zorich (\cite{Zo}).

 Fix an $SL_2(\mathbb{R})$-invariant, ergodic measure $\mu$ on $\Omega\mathcal{M}_g$. Let V be the
restriction of the real Hodge bundle (i.e. the bundle with fibers
$H^1(X,\mathbb{R})$) to the support $M$ of $\mu$. Let $S_t$ be the
lift of the geodesic flow to $V$ via the Gauss-Manin connection.
Then Oseledec's multiplicative ergodic theorem guarantees the
existence of a filtration
$$0\subset V_{\lambda_g}\subset \cdots\subset V_{\lambda_1}=V$$
by measurable vector subbundles with the property that, for almost
all $m\in M$ and all $v\in V_m\backslash\{0\}$ one has
$$||S_t(v)||=\mathrm{exp}(\lambda_it+o(t))$$
where $i$ is the maximal index such that $v$ is in the fiber of
$V_i$ over $m$ i.e.$v\in(V_i)_m$. The numbers $\lambda_i$ for
$i=1,...,k\leq \mathrm{rank}(V)$ are called the \emph{Lyapunov exponents} of
$S_t$. Since $V$ is symplectic, the spectrum is symmetric in the
sense that $\lambda_{g+k}=-\lambda_{g-k+1}$. Moreover, from
elementary geometric arguments it follows that one always has
$\lambda_1=1$.

There is an algebraic interpretation of the sum of certain Lyapunov
exponents:

\begin{theorem}[\cite{KZ97}\cite{Fo02}\cite{BM10}]\label{sumly}If the Variation of Hodge structure (VHS) over the Teichm\"uller curve $C$ contains a sub-VHS $\mathbb{W}$
of rank $2k$, then the sum of the $k$ corresponding to non-negative
Lyapunov exponents equals
$$\overset{k}{\underset{i=1}{\sum}}\lambda^{\mathbb{W}}_i=\frac{2\mathrm{deg} \mathbb{W}^{(1,0)}}{2g(C)-2+|\Delta|},$$
where $\mathbb{W}^{(1,0)}$ is the $(1,0)$-part of the Hodge
filtration of the vector bundle associated with $\mathbb{W}$. In
particular, we have
$$\overset{g}{\underset{i=1}{\sum}}\lambda_i=\frac{2\mathrm{deg}f_*\omega_{S/C}}{2g(C)-2+|\Delta|}.$$
\end{theorem}
Let $L(C)=\overset{g}{\underset{i=1}{\sum}}\lambda_i$ be the sum of
Lyapunov exponents, and put
$\kappa_{\mu}=\frac{1}{12}\overset{k}{\underset{i=1}{\sum}}\frac{m_i(m_i+2)}{m_i+1}$.
Eskin, Kontsevich and Zorich obtain a formula to compute $L(C)$ (for
the Teichm\"{u}ller geodesic flow):
\begin{theorem}[\cite{EKZ}]For the VHS over the Teichm\"uller curve $C$, we have
$$L(C)=\kappa_{\mu}+\frac{\pi^2}{3}c_{area}(C),$$
where $c_{area}(C)$ is the area Siegel-Veech constant corresponding to
$C$.
\end{theorem}
Because the Siegel-Veech constant is non-negative, there is a lower
bound $L(C)\geq k_{\mu}$.
\section{Upper bounds}
Denote by $|\mathcal{L}|$ the projective space of one-dimensional
subspaces of $H^0(X,\mathcal{L})$. For a (projective) $r$-dimension
linear subspace $V$ of $|\mathcal{L}|$, we call $(\mathcal{L},V)$ a
linear series of type $g^r_d$, here $d$ equals $\deg(\mathcal{L})$.

\begin{theorem}[Clifford's theorem \cite{Ha}]\label{clifford}
Let $\mathcal{L}$ be an effective special divisor (i.e.
$h^1(\mathcal{L})\neq 0$) on the curve $X$. Then
$$h^0(\mathcal{L})\leq 1+\frac{1}{2}\mathrm{deg}(\mathcal{L}).$$
Furthermore, equality occurs if and only if either $\mathcal{L}=0$
or $\mathcal{L}=K$ or $X$ is hyperelliptic and $\mathcal{L}$ is a
multiple of the unique linear series of type $g^1_2$ on $X$.
\end{theorem}
Let $C$ be a Teichm\"{u}ller curve lying in
$\Omega\mathcal{M}_g(m_1,...,m_k)$.  Let $P=(p'_1,...,p'_{2g-2})$ be
a permutation of $2g-2$ points

\[
 \overbrace{
   \underbrace{p_1,...,p_1}_\text{$m_1$},...,
   \underbrace{p_k,...,p_k}_\text{$m_k$}
  }^\text{$2g-2$},
\]
here the point $p_i$ is the intersection of the section $D_i$ with the
generic  fiber $F$ of the universal family $f:S\rightarrow C$.
\begin{definition}For any permutation $P=(p'_1,...,p'_{2g-2})$, we denote $H_1(P)=0$, For $j=2,...,g$, $H_j(P)=i$ if $h^0(p'_1+\cdots+p'_{i-1})=j-1$ and $h^0(p'_1+\cdots+p'_i)=j$.
\end{definition}
\begin{corollary}\label{hp}For any permutation $P=(p'_1,...,p'_{2g-2})$, we have $H_j(P)\geq 2j-2$ for $j\geq 2$. When
$j<g$, if the equality holds then  $C$ lies in the hyperelliptic
locus.
\end{corollary}
\begin{proof}By Clifford's Theorem $h^0(p'_1+\cdots+p'_i)\leq
1+\frac{\mathrm{deg}(p'_1+\cdots+p'_i)}{2}$. If $h^0(p'_1+\cdots+p'_i)=j$, then $$H_j(P)=i=\mathrm{deg}(p'_1+\cdots+p'_i)\geq 2(j-1),$$
we have $H_j(P)\geq 2j-2$. If the equality holds then the generic fiber is hyperelliptic curve, $C$ lies in the hyperelliptic locus.
\end{proof}

We also denote by $D'_i$ the section which intersect the gerneral fiber is $p'_i$. We use them to construct vector bundles $f_*\mathcal{O}(D'_1+\cdots+D'_i)$ for $1\leq
i\leq 2g-2$.

\begin{proposition}\label{pf}For any permutation $P=(p'_1,...,p'_{2g-2})$, we can construct a filtration
$$0\subset V'_1\subset V'_2\subset\cdots \subset V'_g=f_*\mathcal{O}(m_1D_1+\cdots+m_kD_k),$$
where $V'_j$ is a rank $j$ vector bundle and
$$V'_j=f_*\mathcal{O}(D'_1+\cdots+D'_{H_j(P)})=\cdots=f_*\mathcal{O}(D'_1+\cdots+D'_{H_{j+1}(P)-1}).$$
Moreover for $j\geq 2$, we have
$$V'_j/V'_{j-1}\subset \mathcal{O}_{D'_{H_j(P)}}(D'_1+\cdots+D'_{H_j(P)})$$
\end{proposition}
\begin{proof}Because
$$h^0(p'_1+\cdots+p'_{H_j(P)})=\cdots= h^0(p'_1+\cdots+p'_{H_{j+1}(P)-1})=j,$$
by lemma \ref{EE}, we have
$$V'_j=f_*\mathcal{O}(D'_1+\cdots+D'_{H_j(P)})=\cdots=f_*\mathcal{O}(D'_1+\cdots+D'_{H_{j+1}(P)-1}).$$
for $j\geq 2$, from the exact sequence
$$0\rightarrow f_*\mathcal{O}(D'_1+\cdots+D'_{H_j(P)-1})\rightarrow f_*\mathcal{O}(D'_1+\cdots+D'_{H_j(P)})\rightarrow \mathcal{O}_{D'_{H_j(P)}}(D'_1+\cdots+D'_{H_j(P)}),$$
we have
$$V'_j/V'_{j-1}\subset \mathcal{O}_{D'_{H_j(P)}}(D'_1+\cdots+D'_{H_j(P)})$$
\end{proof}
\begin{proposition}\label{df}There is a filtration of
$\overset{k}{\underset{i=1}{\oplus}}\overset{m_i}{\underset{j=1}{\oplus}}\mathcal{O}_{D_i}(jD_i)$:
$$0\subset V_1\subset V_2\subset\cdots \subset V_{2g-2}=\overset{k}{\underset{i=1}{\oplus}}\overset{m_i}{\underset{j=1}{\oplus}}\mathcal{O}_{D_i}(jD_i),$$
satisfying:
 1). $V_l/V_{l-1}$ is a summand $\mathcal{O}_{D_j}(dD_j)$,
 2). $\mathrm{deg}(V_l/V_{l-1})$ decreases in $l$.\\
From this filtration we can construct a special permutation
\begin{equation}\label{arrange}
P=(p'_1,p'_2,...,p'_{2g-2}),
\end{equation}
such that
\begin{equation}
V_l=\mathrm{grad}(HN(\mathcal{O}_{D'_1+\cdots+D'_l}(D'_1+\cdots+D'_l))).
\end{equation}
\end{proposition}
\begin{proof}
Consider the degree of each summand and let $V_l/V_{l-1}$ be some summand, we can easily construct this filtration of $\overset{k}{\underset{i=1}{\oplus}}\overset{m_i}{\underset{j=1}{\oplus}}\mathcal{O}_{D_i}(jD_i)$.

If $V_l/V_{l-1}$ equals $\mathcal{O}_{D_j}(dD_j)$, then let $p'_l=p_j$. Thus we construct a special permutation
$$P=(p'_1,p'_2,...,p'_{2g-2})$$

 Because $D^2_j< 0$, we have
$$\mathrm{deg}(\mathcal{O}_{D_j}(D_j))>\cdots>\mathrm{deg}(\mathcal{O}_{D_j}((d-1)D_j))>\mathrm{deg}(\mathcal{O}_{D_j}(dD_j)).$$
There are only $d-1$ $p_j$'s appearing before $p'_l$ in permutation $P$. We have $D'_l=D_j$ and
$$D'_1+\cdots+D'_{l-1}=(d-1)D_j+(\texttt{sum of divisors which not contain } D_j).$$
$$D'_1+\cdots+D'_l=dD_j+(\texttt{sum of divisors which not contain } D_j).$$
By lemma \ref{HN}
$$\mathrm{grad}(HN(\mathcal{O}_{dD_j}(dD_j)))/\mathrm{grad}(HN(\mathcal{O}_{(d-1)D_j}((d-1)D_j)))=\mathcal{O}_{D_j}(dD_j).$$
Combined with lemma \ref{SP}, we have
\begin{equation}\label{grad}
\mathrm{grad}(HN(\mathcal{O}_{\overset{l}{\underset{k=1}{\sum}}D'_k}(\overset{l}{\underset{k=1}{\sum}}D'_k)))/\mathrm{grad}(HN(\mathcal{O}_{\overset{l-1}{\underset{k=1}{\sum}}D'_k}(\overset{l-1}{\underset{k=1}{\sum}}D'_k)))=\mathcal{O}_{D_j}(dD_j).
\end{equation}
By induction we get:
$$V_l=\mathrm{grad}(HN(\mathcal{O}_{D'_1+\cdots+D'_l}(D'_1+\cdots+D'_l))).
$$
\end{proof}
\begin{lemma}We order the numbers in the set
$\{\frac{j}{m_l+1}|1\leq j\leq m_l,1\leq l\leq k\}$
increasing(counted with multiplicity) as $a_1\leq a_2\leq \cdot\cdot\cdot \leq a_{m_1+\cdots+m_k}$. $V_i$ is the vectror bundle constructed in Proposition \ref{df}. Then we have
$$\mathrm{deg}(V_n/V_{n-1})/\mathrm{deg}(\mathcal{L})=-a_n.$$
\end{lemma}
\begin{proof}By Proposition \ref{df}, $\mathrm{deg}(V_n/V_{n-1})$ is the $n$-th largest number of
$\{\mathrm{deg}(\mathcal{O}_{D_l}(jD_l))|1\leq j\leq m_l,1\leq l\leq k\}$. By equation \ref{intersection}, we have
$$\mathrm{deg}(\mathcal{O}_{D_l}(jD_l))/\mathrm{deg}(\mathcal{L})=jD^2_l/\mathrm{deg}(\mathcal{L})=-\frac{j}{m_l+1}.$$

So $-\mathrm{deg}(V_n/V_{n-1})/\mathrm{deg}(\mathcal{L})$ is the $n$-th smallest number of
$\{\frac{j}{m_l+1}|1\leq j\leq m_l,1\leq l\leq k\}$.
\end{proof}
In fact we have obtained an upper bound of the slope of each graded
quotient of the Harder-Narasimhan filtration of $f_*(\omega_{S/C})$
for Teichm\"{u}ller curves:

\begin{theorem}\label{ubhn}For a Teichm\"{u}ller
curve which lies in $\Omega\mathcal{M}_g(m_1,...,m_k)$. We order the numbers in the set
$\{\frac{j}{m_l+1}|1\leq j\leq m_l,1\leq l\leq k\}$
increasing(counted with multiplicity) as $a_1\leq a_2\leq \cdots \leq a_{m_1+\cdots+m_k}$.

Then we have $w_1=1$. For $2 \leq i\leq g$, there are
inequalities:
$$w_i\leq 1-a_{H_i(P)},$$
here $P$ is the special permutation \eqref{arrange} which satisfies $H_i(P)\geq 2i-2$.
\end{theorem}
\begin{proof}By Proposition \ref{df}, there is a filtration of
$\overset{k}{\underset{i=1}{\oplus}}\overset{m_i}{\underset{j=1}{\oplus}}\mathcal{O}_{D_i}(jD_i)$:
$$0\subset V_1\subset V_2\subset\cdots \subset V_{2g-2}=\overset{k}{\underset{i=1}{\oplus}}\overset{m_i}{\underset{j=1}{\oplus}}\mathcal{O}_{D_i}(jD_i),$$
satisfying:
 1). $V_l/V_{l-1}$ is a summand $\mathcal{O}_{D_j}(dD_j)$,
 2). $\mathrm{deg}(V_l/V_{l-1})$ decreases in $l$.\\
From this filtration we also construct a special permutation
$$P=(p'_1,p'_2,...,p'_{2g-2}),$$
such that
$$V_i=\mathrm{grad}(HN(\mathcal{O}_{D'_1+\cdots+D'_i}(D'_1+\cdots+D'_i))).
$$

By Proposition \ref{pf}, use vector bundles $f_*\mathcal{O}(D'_1+\cdots+D'_i)$, we also
construct a filtration
\begin{equation}\label{filtration}
0\subset V'_1\subset V'_2\subset\cdots \subset
V'_g=f_*\mathcal{O}(m_1D_1+\cdots+m_kD_k),
\end{equation}
such that
$$ V'_j=f_*\mathcal{O}(D'_1+\cdots+D'_{H_j(P)})=\cdots=f_*\mathcal{O}(D'_1+\cdots+D'_{H_{j+1}(P)-1}).$$
By equation \eqref{grad}
$$\mathcal{O}_{D'_{H_j(P)}}(D'_1+\cdots+D'_{H_j(P)}))=V_{H_j(P)}/V_{H_j(P)-1}.$$
By Proposition \ref{pf}, for $j\geq 2$,
$$V'_j/V'_{j-1}\subset\mathcal{O}_{D'_{H_j(P)}}(D'_1+\cdots+D'_{H_j(P)})=V_{H_j(P)}/V_{H_j(P)-1}.$$

Now we construct a new filtration:
$$0\subset \mathcal{O}\subset\mathcal{O}\oplus V_{H_2(P)}/V_{H_2(P)-1}\subset\cdots \subset
\mathcal{O}\oplus\overset{g}{\underset{j=2}{\oplus}}V_{H_j(P)}/V_{H_j(P)-1}.$$

This filtration and the filtration \eqref{filtration} satisfy the condition of lemma \ref{control}, so for $j\geq 2$, we have
$$\mu_i(f_*\mathcal{O}(m_1D_1+\cdots+m_kD_k))\leq
\mathrm{deg}(V_{H_i(P)}/V_{H_i(P)-1}).$$
 So for $j\geq 2$, we get
$$w_i=\mu_i(f_*(\omega_{S/C}))/\mathrm{deg}(\mathcal{L})=1+\mu_i(f_*\mathcal{O}(m_1D_1+\cdots+m_kD_k))/\mathrm{deg}(\mathcal{L})\leq
1-a_{H_i(P)}.$$
\end{proof}
\begin{example}\label{ex}
For a Teichm\"{u}ller
curve which lies in $\Omega\mathcal{M}_5(6,1,1)$. We order the numbers in the set
$$\{\frac{1}{7},\frac{2}{7},\frac{3}{7},\frac{1}{2},\frac{1}{2},\frac{4}{7},\frac{5}{7},\frac{6}{7}\}$$
increasing(counted with multiplicity) as $a_1\leq a_2\leq \cdots \leq a_8$. For any special permutation $P$,
then 
$$H_2(P)\geq 2, H_3(P)\geq 4, H_4(P)\geq 6, H_5(P)=8.$$
$$w_2\leq 1-a_{H_2(P)}\leq 1-a_2=\frac{5}{7},$$
$$w_3\leq 1-a_{H_3(P)}\leq 1-a_4=\frac{1}{2},$$
$$w_4\leq 1-a_{H_4(P)}\leq 1-a_6=\frac{3}{7},$$
$$w_5\leq 1-a_{H_5(P)}\leq 1-a_8=\frac{1}{7}.$$

If the general fiber is non hyperelliptic, for any special permutation $P$, by Clifford Theorem or Corollary \ref{pf}, then
$$H_2(P)>2, H_3(P)>4, H_4(P)>6, H_5(P)=8.$$
$$w_2\leq 1-a_{H_2(P)}\leq 1-a_3=\frac{4}{7},$$
$$w_3\leq 1-a_{H_3(P)}\leq 1-a_5=\frac{1}{2},$$
$$w_4\leq 1-a_{H_4(P)}\leq 1-a_7=\frac{2}{7},$$
$$w_5\leq 1-a_{H_5(P)}\leq 1-a_8=\frac{1}{7}.$$
\end{example}
\begin{theorem} \label{main} The sum of Lyapunov exponents of a Teichm\"{u}ller curve in $\Omega\mathcal{M}_g(m_1,...,m_k)$ satisfies the inequality
$$L(C)\leq \frac{g+1}{2}.$$
Furthermore, equality occurs if and only if it lies in the
hyperelliptic locus induced from
$\mathcal{Q}(2k_1,...,2k_n,-1^{2g+2})$ or it is a special
Teichm\"{u}ller curve in $\Omega\mathcal{M}_g(1^{2g-2})$.
\end{theorem}
\begin{proof} By Theorem \ref{ubhn} and Corollary \ref{hp}
\begin{align*}
L(C)=w_1+\overset{g}{\underset{j=2}{\sum}} w_i &\leq g -\overset{g}{\underset{j=2}{\sum}} a_{H_j(P)}\leq g-\overset{g}{\underset{j=2}{\sum}}a_{2j-2} \leq g-\overset{g}{\underset{j=2}{\sum}}(a_{2j-3}+a_{2j-2})/2\\
&=g-\frac{1}{2}\overset{2g-2}{\underset{j=1}{\sum}} a_{j}=g-\frac{1}{2}\overset{k}{\underset{l=1}{\sum}} \overset{m_l}{\underset{j=1}{\sum}} \frac{j}{m_l+1}\\
&=g-\frac{1}{4} \overset{k}{\underset{i=1}{\sum}} m_i=\frac{g+1}{2}.
\end{align*}

When the inequality becomes equal, we have
$a_{H_j(P)}=a_{2j-2}$ and $a_{2j-1}=a_{2j}$. If
$$a_1=a_2=\cdots=a_{2k-2}<a_{2k-1}=a_{2k},$$ then $a_{H_k(P)}=a_{2k-2}$ means $H_k(P)=2k-2$.
Thus by Clifford's theorem, if its generic fibers is not hyperelliptic, then
$a_1=a_2=\cdots=a_{2g-2}$, which means $m_1=\cdots=m_{2g-2}=1$. It is a special Teichm\"{u}ller curve in $\Omega\mathcal{M}_g(1^{2g-2})$.

The hyperelliptic locus in a stratum
$\Omega\mathcal{M}_g(m_1,...,m_k)$ induces from a stratum
$\mathcal{Q}(d_1,...,d_k)$ satisfying $d_1+\cdots+d_n=-4$. A
singularity of order $d_i$ of $q$ give rise to two zeros of degree
$m=d_i/2$ when $d_i$ is even, single zero of degree $m=d+1$ when $d$
is odd.
$$\underset{d_j \mathrm{\ odd}}{\sum} (d_j+1)+\underset{d_j\mathrm{\ even}}{\sum}d_j=2g-2.$$
By the formula of sums for the hyperelliptic locus in \cite{EKZ}\cite{YZ},
$$L(C)=\frac{1}{4}\underset{d_j
\mathrm{\ odd}}{\sum}\frac{1}{d_j+2}\leq \frac{1}{4}\underset{d_j
\mathrm{\ odd}}{\sum} 1=\frac{g+1}{2}.$$ a Teichm\"{u}ller curve in
the hyperelliptic locus satisfies $L(C)=\frac{g+1}{2}$ if and only
if it is induced from $\mathcal{Q}(2k_1,...,2k_n,-1^{2g+2})$.
\end{proof}

\begin{remark}Dawei Chen and M. M\"{o}ller (\cite{CM12}) have constructed a Teichm\"{u}ller curve
$C\in \Omega\mathcal{M}_3(1,1,1,1)$ with $L(C)=2$, but it is not
hyperelliptic: the square tiled surface given by the permutations
$$(\pi_r=(1234)(5)(6789),\pi_{\mu}=(1)(2563)(4897))$$

  They also  have obtained
a bound by using Cornalba-Harris-Xiao's slope inequality
(\cite{Mo13}):
$$L(C) \leq \frac{3g}{(g-1)}\kappa_{\mu}=\frac{g}{4(g-1)}\overset{k}{\underset{i=1}{\sum}}\frac{m_i(m_i+2)}{m_i+1}.$$
When $k\geq 4$, our upper bound better than this upper bound because the right side
$$=\frac{g}{4(g-1)}\overset{k}{\underset{i=1}{\sum}}m_i+\frac{g}{4(g-1)}\overset{k}{\underset{i=1}{\sum}}\frac{m_i}{m_i+1}>\frac{g}{2}+\frac{1}{4}k\frac{1}{2}\geq \frac{g+1}{2}.$$
\end{remark}

We only present an example to explain the general principle on how
to improve the upper bound when we know more information about
Weierstrass semigroups of general fibers.
\begin{corollary}
 A Teichm\"{u}ller curve which lies in the non
hyperelliptic  locus of $\mathcal{M}_4(2,2,1,1)$ satisfies
$$L(C)\leq 13/6.$$
\end{corollary}

\begin{proof} Because $m_1=2,m_2=2,m_3=1,m_4=1$, by Theorem \ref{ubhn}, $a_n$ is the $n$-th number of
$$\{1/3,1/3,1/2,1/2,2/3,2/3\}.$$
 Because the general fiber is non hyperelliptic curve, by Clifford Theorem or Corollary \ref{pf},
 $$H_2(P)>2,H_3(P)>4,H_4(P)=6,$$
 so we can choose the third, the fifth and the sixth
 element of $a_n:1/2,2/3,2/3$. Finally we have
$$L(C)\leq1+\sum^4_{i=2}(1-a_{H_i(P)})=13/6.$$
\end{proof}
This result has appeared in \cite{CM12}.

\section{Applications}
The Harder-Narasimhan filtration always give an upper bound of
degrees of any sub vector bundles, especially those related to the special partial
sums of Lyapunov exponents.

\begin{proposition}\label{pa}If the VHS over the Teichm\"uller curve $C$ contains a sub-VHS $\mathbb{W}$
of rank $2k$, then the sum of the $k$ corresponding non-negative
Lyapunov exponents is the sum of  $w_{i_1},...,w_{i_k}$ (where $i_j$
are different to each other) and satisfies
$$\overset{k}{\underset{i=1}{\sum}}\lambda^{\mathbb{W}}_i\leq 1+\sum^k_{i=2}(1-a_{H_i(P)}).$$
\end{proposition}
\begin{proof}$\mathbb{W}^{(1,0)}$ is summand of $f_*(\omega_{S/C})$ by Deligne's semisimplicity theorem.  The slope $\mu_j(\mathbb{W}^{(1,0)})$ is equal to $\mu_{i_j}(f_*(\omega_{S/C}))$ for some
$j$ by lemma \ref{directsum}, here we can choose $i_j$ such that
each other is different.
$$\overset{k}{\underset{i=1}{\sum}}\lambda^{\mathbb{W}}_i=\frac{2\mathrm{deg} \mathbb{W}^{(1,0)}}{2g(C)-2+|\Delta|}= \frac{\overset{k}{\underset{j=1}{\sum}}\mu_j(\mathbb{W}^{(1,0)})}{\mathrm{deg}(\mathcal{L})}=\overset{k}{\underset{j=1}{\sum}}\mu_{i_j}(f_*(\omega_{S/C}))/\mathrm{deg}(\mathcal{L})=\overset{k}{\underset{j=1}{\sum}}w_{i_j}.$$
By Theorem \ref{ubhn} and $a_i$ decrease in $i$,
$$\overset{k}{\underset{i=1}{\sum}}\lambda^{\mathbb{W}}_i=\overset{k}{\underset{j=1}{\sum}}w_{i_j}\leq 1+\sum^k_{i=2}(1-a_{H_{i_j}(P)})\leq 1+\sum^k_{i=2}(1-a_{H_i(P)}).$$
\end{proof}
For individual Lyapunov exponents, due to the lack of algebraic
interpretation, we will make the following assumption:

\begin{assumption}\label{assumption}$f_*(\omega_{S/C})$ equals $(\overset{h}{\underset{i=1}{\oplus}} L_i) \oplus W$, here $L_i$ are line bundles
such that the $i$-th Lyapunov exponent satisfies the equality:
$$\lambda_i=\Big\{
 \begin{array}{cc}
\mathrm{deg}(L_i)/\mathrm{deg}(\mathcal{L}) &1\leq i\leq h \\
0 &h< i\leq g.
   \end{array}
$$
\end{assumption}
There are many examples satisfying this assumption: triangle groups
\cite{BM10}\cite{Wr2}, square tiled cyclic covers \cite{EKZ11}\cite{FMZ},
square tiled abelian covers \cite{Wr1}, some wind-tree models
\cite{DHL}, and algebraic primitives \cite{BM09}.

Our estimate on the slopes of the Harder-Narasimhan filtration will
give the following upper bound for individual Lyapunov exponents:

\begin{proposition}\label{single}For a Teichm\"{u}ller curve which
satisfies the assumption \ref{assumption} and lies in
$\Omega\mathcal{M}_g(m_1,...,m_k)$. We order the numbers in the set
$\{\frac{j}{m_l+1}|1\leq j\leq m_l,1\leq l\leq k\}$
increasing(counted with multiplicity) as $a_1\leq a_2\leq \cdots \leq a_{m_1+\cdots+m_k}$.

For $2 \leq i\leq g$, the $i$-th Lyapunov exponent
satisfies the inequality:
$$\lambda_i \leq 1-a_{H_i(P)},$$
here $P$ is the special permutation
\eqref{arrange} which satisfies $H_i(P)\geq 2i-2$.
\end{proposition}
\begin{proof}The assumption \ref{assumption} and the lemma \ref{directsum} give us

$$\mathrm{grad}(HN(f_*(\omega_{S/C})))=(\overset{h}{\underset{i=1}{\oplus}} L_i) \oplus \mathrm{grad}(HN(W)),$$

so there are different $j_i$ such that $\lambda_1
=w_{j_1}\geq\lambda_2 =w_{j_2}\geq...\geq\lambda_h=w_{j_h}$. By
Theorem \ref{ubhn}, we have

$$\lambda_i=w_{j_i}\leq w_i\leq 1-a_{H_i(P)}.$$

\end{proof}
The equality can be reached for an algebraic primitive
Teichm\"{u}ller curve lying in the hyperelliptic locus induced from
$\mathcal{Q}(2k_1,...,2k_n,-1^{2g+2})$.

\paragraph{\textbf{Abelian covers}}
 The Lyapunov spectrum has been computed for triangle groups (\cite{BM10}\cite{Wr2}), square tiled cyclic covers (\cite{EKZ11} \cite{FMZ}) and
  square tiled abelian covers (\cite{Wr1}). They all satisfy the assumption
\ref{assumption}. Here we give the description of square tiled
cyclic covers:

Consider an integer $N\geq 1$ and a quadruple of integers
$(a_1,a_2,a_3,a_4)$ satisfying the following conditions:

$$0<a_i\leq N;\quad \mathrm{gcd}(N,a_1,...,a_4)=1;\quad\sum^4_{i=1}a_i\equiv 0 \quad(mod\quad N).$$

Let $z_1,z_2,z_3,z_4\in \mathbb{C}$ be four distinct points. By
$M_N(a_1,a_2,a_3,a_4)$ we denote the closed connected nonsingular
Riemann surface obtained by normalization of the one defined by the
equation

$$w^N=(z-z_1)^{a_1}(z-z_2)^{a_2}(z-z_3)^{a_3}(z-z_4)^{a_4}.$$

Varying the cross-ratio $(z_1,z_2,z_3,z_4)$ we
obtain the moduli curve $\mathcal{M}_{(a_i),N}$. As an abstract
curve it is isomorphic to $\mathcal{M}_{0,4}\simeq
\mathbb{P}^1-\{0,1,\infty\}$; more strictly speaking, it should be
considered as a stack. The canonical generator $T$ of the group of
deck transformations induces a linear map $T^*:H^{1,0}(X)\rightarrow
H^{1,0}(X)$. $H^{1,0}(X)$ admits a splitting into a direct sum of
eigenspaces $V^{1,0}(k)$ of $T^*$ and satisfies the assumption
\ref{assumption}. (cf. Theorem $2$ in \cite{EKZ11})

For even $N$, $M_N(N-1,1,N-1,1)$ has Lyapunov spectrum
(\cite{EKZ11}):

$$\{\frac{2}{N},\frac{2}{N},\frac{4}{N},\frac{4}{N},...,\frac{N-2}{N},\frac{N-2}{N},1\}.$$

\begin{corollary} $M_N(N-1,1,N-1,1)$ lies
in the hyperelliptic locus which induced from $\mathcal
{Q}(N-2,N-2,-1^{2N})$.
\end{corollary}
\begin{proof}
By  the genus formular
$g=N+1-\frac{1}{2}\sum^4_{i=1}\mathrm{gcd}(a_i,N)$ and
$$L(C)=\frac{2}{N}+\frac{2}{N}+\frac{4}{N}+\frac{4}{N}+\cdots+\frac{N-2}{N}+\frac{N-2}{N}+1=\frac{g+1}{2}.$$
By the Theorem \ref{main}, $M_N(N-1,1,N-1,1)$ lies
in the hyperelliptic locus which induced from $\mathcal
{Q}(N-2,N-2,-1^{2N})$.
\end{proof}
\paragraph{\textbf{Algebraic primitives}} The variation of Hodge structures
over a Teichm\"{u}ller curve decomposes into sub-VHS
\begin{equation}\label{VHS}
R^1f_*\mathbb{C}=(\oplus^r_{i=1} \mathbb{L}_i)\oplus \mathbb{M}
\end{equation}
Here $\mathbb{L}_i$ are rank-2 subsystems, maximal Higgs
$\mathbb{L}^{1,0}_1\simeq \mathcal{L}$ for $i=1$, non-unitary but
not maximal Higgs  for $i\neq 1$ (\cite{Mo11}). It is obvious that
the Teichm\"{u}ller curve satisfies the assumption \ref{assumption}
if $r\geq g-1$.

If $r=g$, it is called algebraic primitive Teichm\"{u}ller curves.
 We know there are only finite algebraic primitive Teichm\"{u}ller
 curves
in the stratum $\Omega\mathcal{M}_3(3,1)$ by M\"{o}ller and
Bainbridge in \cite{BM09}, and they conjecture that the algebraic
primitive Teichm\"{u}ller curves in each stratum is finite
(\cite{Mo13}).

\begin{corollary}Algebraic
primitive Teichm\"{u}ller
 curves
in the stratum $\Omega\mathcal{M}_3(3,1)$ has Lyapunov spectrum
$\{1,\frac{2}{4},\frac{1}{4}\}$.
\end{corollary}
\begin{proof}
 By Proposition \ref{nonva}.
\end{proof}
\paragraph{\textbf{Wind-tree models}}A wind-tree model or the infinite
billiard table is defined as:
$$T(a,b):=\mathbb{R}^2\setminus\underset{m,n\in\mathbb{Z}}{\bigcup}[m,m+a]\times[n,n+b]$$
with $0<a,b<1$. Denote by $\phi^{\theta}_t:T(a,b)\rightarrow T(a,b)$
the billiard flow: for a point $p\in T(a,b)$, the point
$\phi^{\theta}_t$ is the position of a particle after time $t$
starting from position $p$ in direction $\theta$.

\begin{theorem}[\cite{DHL}]Let $d(.,.)$ be the Euclidean distance on $\mathbb{R}^2$.\begin{itemize}
 \item (Case 1) If
$a$ and $b$ are rational numbers or can be written as
$1/(1-a)=x+y\sqrt{D},1/(1-b)=(1-x)+y\sqrt{D}$ with $x,y\in
\mathbb{Q}$ and $D$ a positive square-free integer then for Lebesgue
almost all $\theta$ and every point $p$ in $T(a,b)$.
\item (Case 2) For Lebesgue-almost all $(a,b)\in (0,1)^2$, Lebesgue-almost
all $\theta$ and every point $p$ in $T(a,b)$ (with an infinite
forward orbit):
\begin{equation}\label{wind}
\underset{T\rightarrow +\infty}{\mathrm{lim}\,} \mathrm{sup}\frac{\mathrm{log}\,
d(p,\phi^{\theta}_T(p))}{\mathrm{log}\,T}=\frac{2}{3}.
\end{equation}\end{itemize}
\end{theorem}
We are interested in the case $1$ because it is related to
Teichm\"{u}ller curves. By the Katok-Zemliakov construction, the
billiard flow can be replaced by a linear flow on a(non
compact)translation surface which is made of four copies of $T(a,b)$
that we denote $X_{\infty}(a,b)$. The surface $X_{\infty}(a,b)$ is
$\mathbb{Z}^2$ -periodic and we denote by $X(a,b)$ the quotient of
$X_{\infty}(a,b)$ under the $\mathbb{Z}^2$ action.

The surface $X(a,b)$ is a covering (with Deck group
$\mathbb{Z}/2\times \mathbb{Z}/2$) of the genus $2$ surface
$L(a,b)\in \Omega\mathcal{M}_2(2)$ which is called L-shaped surface
(\cite{Ca} \cite{Mc03}). The orbit of $X(a,b)$ for the
Teichm\"{u}ller flow belongs to the moduli space
$\Omega\mathcal{M}_5(2,2,2,2)$.

The Teichm\"{u}ller curve generated by the surface $X(a,b)$
satisfies the assumption \ref{assumption} because there is an
$SL_2(\mathbb{R})$-equivalent splitting of the Hodge bundle. Its
Lyapunov spectrum is
$\{1,\frac{2}{3},\frac{2}{3},\frac{1}{3},\frac{1}{3}\}$, the
equation \eqref{wind} is equivalence to say that
$\lambda_2=\frac{2}{3}$ (\cite{DHL}).

\begin{corollary}
$X(a,b)$ is lies in the
hyperelliptic locus which induced from $\mathcal {Q}(4,4,-1^{12})$.
\end{corollary}
\begin{proof}
 By the Theorem \ref{main}, $X(a,b)$ is lies in the
hyperelliptic locus which induced from $\mathcal {Q}(4,4,-1^{12})$,
because $L(C)$ equal $\frac{g+1}{2}$.
\end{proof}
In fact, a Teichm\"{u}ller curve which satisfies the assumption
\ref{assumption} and lies in $\Omega\mathcal{M}_5(2,2,2,2)$
satisfies $\lambda_2\leq\frac{2}{3}$ by the Proposition
\ref{single}.
\paragraph{\textbf{Non-varying strata}}
Recently, there are many progresses about the phenomenon that the
sum of Lyapunov exponents is non varying in some strata
(\cite{CM12}\cite{CM14}\cite{YZ}). The following proposition is a an
immediate corollary of the Theorem \ref{YZ}.
\begin{proposition}\label{nonva}For a Teichm\"{u}ller curve which
satisfies the assumption \ref{assumption} and lies in hyperelliptic
loci or one of the following strata:\\
$\overline{\Omega\mathcal{M}}_3(4),\overline{\Omega\mathcal{M}}_3(3,1),\overline{\Omega\mathcal{M}}^{odd}_3(2,2),\overline{\Omega\mathcal{M}}_3(2,1,1)$\\
$\overline{\Omega\mathcal{M}}_4(6),\overline{\Omega\mathcal{M}}_4(5,1),\overline{\Omega\mathcal{M}}^{odd}_4(4,2),\overline{\Omega\mathcal{M}}^{non-hyp}_4(3,3),\overline{\Omega\mathcal{M}}^{odd}_4(2,2,2),\overline{\Omega\mathcal{M}}_4(3,2,1)$\\
$\overline{\Omega\mathcal{M}}_5(8),\overline{\Omega\mathcal{M}}_5(5,3),\overline{\Omega\mathcal{M}}^{odd}_5(6,2)$.\\
The $i$-th Lyapunov exponent $\lambda_i$ equals the $w_i$ which is
computed in the Theorem \ref{YZ}.
\end{proposition}
\begin{proof}The assumption \ref{assumption} and the lemma \ref{directsum} give us

$$\mathrm{grad}(HN(f_*(\omega_{S/C})))=(\overset{h}{\underset{i=1}{\oplus}} L_i) \oplus \mathrm{grad}(HN(W)).$$

We have constructed the Harder-Narasimhan filtration with $w_i>0$ in
\cite{YZ}. If $h<g$, then $\mathrm{deg}(W)=0$ by the assumption \ref{assumption}. Using lemma \ref{directsum}, we get

$$0=\frac{\mathrm{deg}(W)}{\mathrm{deg}(\mathcal{L})}=\frac{\overset{g-h}{\underset{i=1}{\sum}}
\mu_i(W)}{\mathrm{deg}(\mathcal{L})}=\frac{\overset{g-h}{\underset{i=1}{\sum}}
\mu_{j_i}(f_*(\omega_{S/C}))}{\mathrm{deg}(\mathcal{L})}=\overset{g-h}{\underset{i=1}{\sum}}
w_{j_i}>0.$$

It is contradiction! Thus we have
$\mathrm{grad}(HN(f_*(\omega_{S/C})))=\overset{g}{\underset{i=1}{\bigoplus}}
L_i$ and $\lambda_i=w_i$.
\end{proof}

\paragraph{\textbf{Hyperelliptic loci}}It has been shown in \cite{EKZ11} that the ''stairs''
square tiled surface $S(N)$ satisfies the assumption
\ref{assumption} and belongs to the hyperelliptic
 connected component $\overline{\Omega\mathcal{M}}^{hyp}_g(2g-2)$, for $N=2g-1$
 or
   $\overline{\Omega\mathcal{M}}^{hyp}_g(g-1,g-1)$, for $N=2g$ .

\begin{remark}
  The  Proposition \ref{nonva} also implies that
 the Lyapunov spetrum of the Hodge bundles over the corresponding
 arithmetic  Teichm\"{u}ller curves  is
$$\Lambda Spec=\{
 \begin{array}{cc}
\frac{1}{N},\frac{3}{N},\frac{5}{N},...,\frac{N}{N}& N=2g-1 \\
\frac{2}{N},\frac{4}{N},\frac{6}{N},...,\frac{N}{N}& N=2g.
   \end{array}
$$
Which has been shown in \cite{EKZ11} by using the fact $S(N)$ is
quotient of $M_N(N-1,1,N-1,1)$ (resp. $M_{2N}(2N-1,1,N,N)$) for $N$
is even (resp. odd).
\end{remark}
\paragraph{\textbf{Prym varieties}} McMullen, use Prym eigenforms, has constructed infinitely many
primitive Teichm\"{u}ller curves for $g=2,3$ and $4$ (\cite{Mc06a}).
Let $W_{D}(6)$ be the Prym Teichm\"{u}ller curves in
$\Omega\mathcal{M}_4$. It has VHS decomposition:
$$R^1f_*\mathbb{C}=(\mathbb{L}_1\oplus \mathbb{L}_2)\oplus \mathbb{M}.$$
So it map to curves $W^X_D$ in the Hilbert modular surface
$X_D=\mathbb{H}^2/SL(\mathcal{O}_D\oplus\mathcal{O}_D^{\vee})$.
\begin{remark}
The proposition \ref{pa} tells us that the number
$\mathrm{deg}(\mathbb{L}^{1,0}_2)/\mathrm{deg}(\mathcal{L})$ equals one of the numbers
$\{\frac{4}{7},\frac{2}{7},\frac{1}{7}\}$. In fact it has been shown
that $W^X_D$ is the vanishing locus of a modular form of weight
$(2,14)$, so $\mathrm{deg}(\mathbb{L}^{1,0}_2)/\mathrm{deg}(\mathcal{L})$ is
$\frac{1}{7}$. (\cite{Mo13}\cite{We14})
\end{remark}


\begin{thebibliography}{sd}
 \bibitem[BaMo]{BM09}M.Bainbridge and M.M\"{o}ller, \textit{The Deligne-Mumford compactification of
the real multiplication locus and Teichm\"{u}ller curves in genus
three}. Acta Mathematica Acta Mathematica \textbf{208}(1) (2012), 1--92.
  \bibitem[BoMo]{BM10}I.Bouw and M.M\"{o}ller, \textit{Teichm\"{u}ller cuves, triangle groups and Lyapunov exponents}.  Ann of math.(2) \textbf{172}(1) (2010), 139--185.
 \bibitem[Ca]{Ca}K.Calta, \textit{Veech surfaces and complete periodicity in genus two}. J. Amer. Math. textbf{228} (2011), Soc. 17 (4)(2004), 871--908.
   \bibitem[CM12]{CM12}D.Chen and M.M\"{o}ller, \textit{Non-varying sum of Lyapunov exponents of Abelian differentials in low genus}. Geom. Topol. \textbf{16} (2012), no. 4, 2427--2479.
 \bibitem[CM14]{CM14}D.Chen and M.M\"{o}ller, \textit{Quadratic differentials in low genus: exceptional and non-varying}. Ann. Sci. \'Ecole Norm. Sup.(4) \textbf{47} (2014), no. 2, 309--369.
\bibitem[DHL]{DHL}V.Delecroix, P.Hubert and S.Leli\'{e}vre, \textit{Diffusion for the periodic wind-tree model}.  Ann. Sci. \'Ecole Norm. Sup. (4)\textbf{47} (2014), no. 6, 1085--1110.
    \bibitem[EKZ11]{EKZ11}A.Eskin, M.Kontsevich and A.Zorich, \textit{Lyapunov spectrum of
square-tiled cyclic covers}. J. Mod. Dyn. 5 (2011), no. 2, 319--353.
  \bibitem[EKZ]{EKZ}A.Eskin, M.Kontsevich and A.Zorich, \textit{Sum of Lyapunov exponents of the Hodge bundle with respect to the Teichm\"{u}ller geodesic flow}. Publ. Math. Inst. Hautes \'Etudes, \textbf{120}(1)(2014), 207--333.
  \bibitem[EMZ03]{EMZ03}A.Eskin, H.Masur and A.Zorich, \textit{Moduli spaces of Abelian differentials:the principal boundary,counting problems and the Siegel-Veech constats}. Publ. Math. Inst. Hautes \'Etudes, \textbf{97} (2003), 61--179.
\bibitem[Fo02]{Fo02}Giovanni Forni, \textit{Deviation of ergodic averages for area-preserving ows on surfaces of higher genus}.  Ann. of Math. (2) \textbf{155} (2002), no. 1, 1--103.
\bibitem[FMZ]{FMZ}Giovanni Forni, Carlos Matheus and Anton Zorich, \textit{Square-tiled cyclic covers}.
J. Mod. Dyn. \textbf{5} (2011), no.2, 285--318.
\bibitem[HL]{HL}D.Huybrechts and M.Lehn, \textit{The Geometry of Moduli spaces of
sheaves}. Aspects of Mathematics 31, Friedr. Vieweg and Sohn,
Braunschweig (1997).
\bibitem[Ha]{Ha}R.Hartshorne, \textit{Algebraic geometry}. GTM 52 (1977) Springer, New
York.
 \bibitem[KZ97]{KZ97}M.Kontsevich and A.Zorich, \textit{Lyapunov exponents and Hodge theory}. \href{http://arxiv.org/abs/hep-th/9701164}{arxiv.org/abs/hep-th/9701164}
 \bibitem[KZ03]{KZ03}M.Kontsevich and A.Zorich, \textit{Connected
components of the moduli spaces of Abelian differentials with
prescribed singularities}. Invent. Math. \textbf{153}(3) (2003), 631--678.
  \bibitem[Mc03]{Mc03} C.T.McMullen,  \textit{Billiards and Teichm\"{u}ller curves on Hilbert modular surfaces}.  J. Amer. Math. Soc. \textbf{16}(2003), no. 4,  857--885.
  \bibitem[Mc06a]{Mc06a} C.T.McMullen, \textit{Prym varieties and Teichm\"{u}ller curves}. Duke Math.J. \textbf{133} (2006), no. 3, 569--590.
  \bibitem[Mo11]{Mo11}M.M\"{o}ller, \textit{Shimura and Teichm\"{u}ller curves}. J. Mod. Dyn. \textbf{5} (2011), no. 1, 1--32.
 \bibitem[Mo]{Mo}M.M\"{o}ller, \textit{Prym covers, theta functions and Kobayashi
curves in Hilbert modular surfaces}.
Amer. J. Math. \textbf{136} (2014), no. 4, 995--1021.
  \bibitem[Mo13]{Mo13}M.M\"{o}ller,\textit{Teichm\"{u}ller curves, mainly from the view point of algebraic geometry}. IAS/Park City Math. Ser. \textbf{20} (2013), 267--318.
\bibitem[We14]{We14}C.Wei{\ss}, \textit{Twisted Teichm\"{u}ller curves}. Lecture Notes in Mathematics \textbf{2104}, Springer, Cham., 2014
\bibitem[Wr12]{Wr1}A.Wright, \textit{Schwarz triangle mappings and Teichm¨¹ller curves: abelian square-tiled surfaces}.
J. Mod. Dyn. \textbf{6}(3) (2012), 405--426.
\bibitem[Wr13]{Wr2}A. Wright, \textit{Schwarz triangle mappings and Teichm¨¹ller curves: the Veech-Ward-Bouw-M\"oller curves}. Geom. Funct. Anal. \textbf{23} (2013), no. 2, 776--809.
\bibitem[Yu]{Yu}F.Yu, \textit{Eigenvalues of curvature, Lyapunov exponents and Harder-Narasimhan filtrations}.,
\href{http://arxiv.org/abs/1408.1630}{arxiv.org/abs/1408.1630}
 \bibitem[YZ]{YZ}F.Yu and K.Zuo, \textit{Weierstrass filtration on Teichm\"{u}ller curves and Lyapunov exponents}. J. Mod. Dyn. \textbf{7}(2) (2013), 209--237.
 \bibitem[Zo]{Zo}A.Zorich, \textit{Flat surfaces}. In Frontiers in Number Theory, Physics and Geometry.
Volume 1:On random matrices, zeta functions and dynamical systems,
pages 439--586. Springer-Verlag, Berlin, 2006.
\end{thebibliography}
\end{document}